\documentclass[12pt]{article}
\usepackage[utf8]{inputenc}
\usepackage{amsmath,amssymb,amsfonts,amsthm,mathrsfs,mathtools}
\newtheorem{theorem}{Theorem}
\title{On compositions of natural numbers}
\author{\it Douglas E. Iannucci}
\date{\small\sc Pungenday, Confusion 7, Year of Our Lady of Discord 3186}

\begin{document}

\maketitle
\begin{abstract}
In this expository note, we introduce the reader to compositions of a natural number, e.g., $2+1+2+1+7+1$ is a composition of~14, and $1+2$ and $2+1$ are two different compositions of~3. We discuss some simple restricted forms of compositions, e.g., $23+17+33$ is a composition of~73 into three odd parts.  We derive formulas that count the number of so restricted forms of compositions of a natural number~$n$, and we conclude with a brief general discussion of the topic.
\end{abstract}

\section{Introduction}\label{intro}
Let~$n$ be a natural number. A {\bf composition of}~$\boldsymbol{n}$ is an ordered sequence of natural numbers whose sum is~$n$. The notation used to indicate a composition may vary from author to author. Here, we take a classical approach, and write a composition of~$n$ as an equation that expresses~$n$ as a sum, e.g., two compositions of~37:
\begin{align*}
37&=11+8+3+7+2+6\\
37&=1+5+1+1+1+17+1+2+2+4\end{align*}
The addends themselves are called the {\bf parts} of the composition. In the first composition of~37 given above, the parts are 11, 8, 3, 7, 2, and~6, in that order. Thus, by our convention, we write a composition of~$n$ in the form
\begin{equation}\label{acform}
n=a_1+a_2+\cdots+a_k,
\end{equation}
where $a_1$, $a_2$, \dots, $a_k$ denote the parts of the composition. It is clear from~\eqref{acform} that a composition of~$n$ may contain as few as one part, that being~$n$ itself, or as many as~$n$ parts, all being unity. 

We stress here that the sequence of parts is {\bf ordered}. For example,
\begin{align*}
17&=4+1+3+4+1+1+3\\
17&=1+1+3+4+3+4+1\end{align*}
give two different compositions of~17. 

Clearly, unity has exactly one composition, $1=1$. Then,
\begin{align*}
2&=2 & 2&=1+1\end{align*}
are the two compositions of~2. How many compositions are there of~3, 4, and 5 respectively? Observe:
\begin{align*}
3&=3 & 3&=1+2 \\
3&=2+1 & 3&=1+1+1\end{align*}
Thus, there are exactly~4 compositions of~3. There are exactly~8 compositions of~4, 
\begin{align*}
4&=4 & 4&=1+3\\
4&=3+1 & 4&=1+2+1\\
4&=2+2 & 4&=1+1+2\\
4&=2+1+1 & 4&=1+1+1+1\end{align*}
and~16 compositions of~5,
\begin{align*}
5&=5 & 5&=1+4\\
5&=4+1 & 5&=1+3+1\\
5&=3+2 & 5&=1+2+2\\
5&=3+1+1 & 5&=1+2+1+1\\
5&=2+3 & 5&=1+1+3\\
5&=2+2+1 & 5&=1+1+2+1\\
5&=2+1+2 & 5&=1+1+1+2\\
5&=2+1+1+1 & 5&=1+1+1+1+1\end{align*}
By now, one likely intuits that there are exactly $2^{n-1}$ compositions of~$n$. This is correct. It is proved easily by induction; we do this in \S~\ref{acsimple}. Note that $2^{n-1}$ is sequence {\tt A000079} in the {\it OEIS}~\cite{oeis}.
 \begin{table}
 \centering
 \begin{tabular}{| l | l | l|} 
 \hline
 Counting function: & Restriction on the compositions of~$n$: & \S:\\
 \hline
 $s_{r,m}(n)$    & $a\equiv r\pmod{m}$ for all parts~$a$. & \ref{recerrm}\\ 
 $s_{r,m,k}(n)$  & Same as $s_{r,m}(n)$, but with exactly~$k$ parts. & \ref{rmksec} \\
 $t_q(n)$        & $a\le q$ for all parts~$a$.  & \ref{restpart} \\
 $u_q(n)$        & $a\ge q$ for all parts~$a$. & \ref{restpart} \\
 $v_{p,q}(n)$     & $p\le a\le q$ for all parts~$a$. & \ref{restpart} \\
 $t_{q,k}(n)$        & Same as $t_{q}(n)$, but with exactly~$k$ parts. & \ref{sizenumber}  \\
 $u_{q,k}(n)$        & Same as $u_{q}(n)$, but with exactly~$k$ parts. & \ref{sizenumber} \\
 $v_{p,q,k}(n)$     & Same as $v_{p,q}(n)$, but with exactly~$k$ parts. & \ref{sizenumber} \\ \hline
  \end{tabular}
  \caption{Counting functions discussed in this note.}
  \label{listofc}
  \end{table}

In this note we derive general formulas for counting the number of compositions with additional restrictions imposed. For example, we'll see in \S~\ref{rmksec} that the number of compositions of~73 comprising exactly three odd parts is the number of the Beast, 666. In Table~\ref{listofc}, we summarize these counting functions. 

\section{The number of compositions of~$\boldsymbol{n}$}\label{acsimple}
Let $C(n)$ denote the number of compositions of~$n$. We prove the well known formula $C(n)=2^{n-1}$ by induction on~$n$, in a natural way suggested by the lexicographic order in which the compositions of 1--5 are given in \S~\ref{intro}.

\begin{theorem}\label{01}
There are precisely $2^{n-1}$ compositions of the natural number~$n$. That is, $C(n)=2^{n-1}$ for all natural numbers~$n$.
\end{theorem}
\begin{proof}
Clearly $C(1)=1$. Consider $n>1$. In any composition of~$n$, the first part, denoted by~$a_1$ in~\eqref{acform}, has the property $1\le a_1\le n$. If $a_1=n$, then $n=a_1$ is itself a composition of~$n$. Otherwise $k>1$. 

Consider, then, each fixed value of~$a_1$, $1\le a_1<n$. Let $a=a_1$. There are exactly $C(n-a)$ compositions
$$n-a=a_2+a_3+\cdots+a_k.$$
Thus
$$C(n)=1+\sum_{a=1}^{n-1}C(n-a).$$
By induction hypothesis,
$$C(n)=1+\sum_{a=1}^{n-1}2^{a-1}=1+(2^{n-1}-1)=2^{n-1},$$
hence the result.
\end{proof}

We seek to generalize. How may we count the number of compositions of~$n$ where all the parts are odd? How about when the number of parts is fixed? How about when the parts are bounded?

It is easier to obtain such generalizations if we use a slightly different method of induction. Before proving the general formula, we first illustrate this method by providing an alternative proof to Theorem~\ref{01}.

It begins the same way: $C(1)=1$, hence we consider $n>1$. A composition of~$n$, as in~\eqref{acform}, either has the property that $a_1=1$ or $a_1>1$. Let~$A$ denote the set of compositions of~$n$ such that $a_1=1$, and let~$B$ denote the set of compositions of~$n$ such that $a_1>1$. Note that~$A$ and~$B$ are disjoint. Thus, 
\begin{equation}\label{decab}
C(n)=\#A+\#B,
\end{equation}
where $\#S$ denotes the number of elements in a finite set~$S$.

For all compositions in~$A$, we may write
\begin{equation}\label{1plus}
n-1=a_2+a_3+\cdots+a_k,
\end{equation}
while for all compositions in~$B$, we may write
\begin{equation}\label{minus1}
n-1=(a_1-1)+a_2+a_3+\cdots+a_k.
\end{equation}
The sum in~\eqref{1plus} is a  composition of~$n-1$; therefore $\#A=2^{n-2}$ by induction hypothesis. The sum in~\eqref{minus1} is also a composition of~$n-1$, where $a_1-1$ is the first part; thus, similarly, $\#B=2^{n-2}$. Hence by~\eqref{decab} $C(n)=2^{n-2}+2^{n-2}=2^{n-1}$.

The key step here was establishing the recurrence 
$$C(n)=C(n-1)+C(n-1),$$
which in this case simplifies to $C(n)=2C(n-1)$. We can alter this key step slightly to answer the question: how many  compositions on~$n$ are there such that all the parts are all odd? Let us denote this number by $D(n)$. Clearly $D(1)=1$, and $D(2)=1$ because $2=1+1$ is the only way to compose~2 into odd parts. Thus, we may assume $n>2$. Again, we may let~$A$ denote the set of compositions of~$n$ with all odd parts, such that~$a_1=1$. Again, we may let~$B$ denote the set of compositions of~$n$ with all odd parts, such that~$a_1>1$ (hence $a_1>2$). Again, $A$ and~$B$ are disjoint, so that $D(n)=\#A+\#B$.

For all compositions in~$A$, we may write
\begin{equation}\label{1plusodd}
n-1=a_2+a_3+\cdots+a_k,
\end{equation}
and for all compositions in~$B$, we may write
\begin{equation}\label{minus2}
n-2=(a_1-2)+a_2+a_3+\cdots+a_k.
\end{equation}
Similarly as before, the sum in~\eqref{1plusodd} is a composition of $n-1$ into odd parts. Likewise, the sum in~\eqref{minus2} is a composition of $n-2$ into odd parts, where $a_1-2$ is the first part. Thus $\#A=D(n-1)$, $\#B=D(n-2)$, hence
\begin{equation*}
D(n)=D(n-1)+D(n-2).
\end{equation*}
Therefore
$$D(1)=D(2)=1,\qquad D(n)=D(n-1)+D(n-2)\quad(n>2).$$
This defines $D(n)$ as the Fibonacci sequence ({\tt A000045} in the {\it OEIS\/}); i.e.,
$$D(n)=F_n.$$
We give here the compositions of~$n$ into odd parts for $n=3$, 4, and~5:
\begin{align*}
3&=3     &    4&=3+1   &   5&=5\\
3&=1+1+1 &    4&=1+3  &   5&=3+1+1\\
  &       &    4&=1+1+1+1& 5&=1+3+1\\
   &      &     &        & 5&=1+1+3\\
  &       &     &        & 5&=1+1+1+1+1\end{align*}
 
\section{The recurrence for $\boldsymbol{s_{r,m}(n)}$}\label{recerrm}
We introduce here the notation $s_{r,m}(n)$ to represent the number of compositions of~$n$ into parts, all of which are congruent to~$r$ modulo~$m$. Thus we assume that $0<r\le m$. Furthermore, without loss of generality, we may assume that~$r$ and~$m$ are relatively prime. For, otherwise $m=d\mu$ and $r=d\rho$ for natural numbers~$\mu$, $\rho$, and~$d$, where $d>1$ and $\gcd(\mu,\rho)=1$. Thus, all the parts $a_j$, as in~\eqref{acform}, of a composition of~$n$ that is counted by $s_{r,m}(n)$ have the form
$$a_j=mb_j+r=d(\mu b_j+\rho).$$
Letting $c_j=\mu b_j+\rho$, such a composition on~$n$ has the form
$$n=d\left(c_1+c_2+\cdots +c_k\right),$$
hence $n=d\nu$ for some natural number~$\nu$. Hence
$$\nu=c_1+c_2+\cdots +c_k,$$
which is a composition of~$\nu$ counted by $s_{\rho,\mu}(\nu)$. Therefore $s_{m,r}(n)=s_{\rho,\mu}(\nu)$.

According to our notation, Theorem~\ref{01} states that $s_{1,1}(n)=2^{n-1}$. In \S~\ref{acsimple} we showed that $s_{1,2}(n)=F_n$. By mirroring our remarks in \S~\ref{acsimple}, we can establish a recursive definition for $s_{r,m}(n)$ in general:

\begin{theorem}\label{rm}
Given $0<r\le m$ and $\gcd(r,m)=1$, if $1\le n\le m$ then
\begin{equation}\label{thmsrm}
s_{r,m}=\begin{cases}1,&\text{if $r\mid n$,}\\
                       0,&\text{if $r\nmid n$.}\end{cases}
                       \end{equation}
If $n>m$ then $s_{r,m}(n)=s_{r,m}(n-r)+s_{r,m}(n-m)$.
\end{theorem}
\begin{proof}
The only possible part~$a$, such that $a\le m$ and $a\equiv r\pmod{m}$, is~$a=r$. Thus the only possible  compositions of~$n$, with parts congruent to~$r$ modulo~$m$, such that $n\le m$ are 
$$ n=r,\quad n=r+r,\quad n=r+r+r,\quad\dots,\quad n=\underbrace{r+r+\cdots+r}_{\hbox{\text{$[m/r]$ parts}}}.$$
Thus~\eqref{thmsrm} holds when $1\le n\le m$. 

Otherwise $m>n$. For any composition of~$n$ counted by $s_{r,m}(n)$, the first part~$a_1$ has either the property $a_1=r$ or $a_1>r$. Let~$A$ denote those compositions such that $a_1=r$, and let~$B$ denote those such that $a_1>r$ (hence $a_1\ge m+r$). Then~$A$ and~$B$ are disjoint, hence
\begin{equation}\label{abagain}
s_{r,m}(n)=\#A+\#B.
\end{equation}
Each composition in~$A$ may be written as
\begin{equation}\label{amore}
n-r=a_2+a_3+\cdots+a_k,
\end{equation}
while each composition in~$B$ has the form
\begin{equation}\label{bmore}
n-m=\left(a_1-m\right)+a_2+\cdots+a_k.
\end{equation}
The sum in~\eqref{amore} is a composition of~$n-r$ such that all parts are congruent to~$r$ modulo~$m$. The sum in~\eqref{bmore} is a composition of~$n-m$ such that all parts are congruent to~$r$ modulo~$m$. Therefore $\#A=s_{r,m}(n-r)$ and $\#B=s_{r,m}(n-m)$, hence by~\eqref{abagain}
$$s_{r,m}(n)=s_{r,m}(n-r)+s_{r,m}(n-m).$$
\end{proof}

Note that Theorem~\ref{rm} subsumes Theorem~\ref{01}. Indeed, by Theorem~\ref{rm},
$$s_{1,1}(1)=1, \qquad s_{1,1}(n)=2s_{1,1}(n-1)\quad(n>1),$$
which, as mentioned before, defines the sequence $2^{n-1}$. Theorem~\ref{rm} then implies
$$s_{1,2}(1)=s_{1,2}(2)=1,\qquad s_{1,2}(n)=s_{1,2}(n-1)+s_{1,2}(n-1)\quad(n>2),$$
which, as seen, defines the Fibonacci sequence~$F_n$. Furthermore, we have
$$s_{1,3}(1)=s_{1,3}(2)=s_{1,3}(3)=1,\qquad s_{1,3}(n)=s_{1,3}(n-1)+s_{1,3}(n-3)\quad(n>3),$$
which defines Narayana's cows sequence ({\tt A000930} in the {\it OEIS\/}). Similarly,
\begin{gather*}
s_{2,3}(1)=0,\qquad s_{2,3}(2)=1,\qquad s_{2,3}(3)=0,\\
s_{2,3}(n)=s_{2,3}(n-2)+s_{2,3}(n-3)\quad(n>3)
\end{gather*}
defines the Padovan sequence ({\tt A000931} in the {\it OEIS\/}), where the sequence elements $a(n)$ in {\tt A000931} are enumerated such that $s_{2,3}(n)=a(n+1)$.

\begin{table}
\centering
 \begin{tabular}{|lll|lll|} 
 \hline
 $m$ & $s_{1,m}(n)$ &in {\it OEIS} seq. &  $m$ & $s_{1,m}(n)$ &in {\it OEIS} seq.\\
 \hline
 1 & $a(n)$ &{\tt A000079} & 5 & $a(n-1)$ &{\tt A003520} \\ 
 2 & $a(n)$ &{\tt A000045} & 6 & $a(n-1)$ &{\tt A005708} \\
 3 & $a(n-1)$ &{\tt A000930} & 7 & $a(n-1)$ &{\tt A005709} \\
 4 & $a(n)$ &{\tt A003269} & 8 & $a(n-1)$ &{\tt A005710} \\
 \hline
 \end{tabular}
 \caption{Some sequences of form $s_{1,m}(n)$.}
 \label{1m}
\end{table}

Several sequences of the form $s_{1,m}(n)$ (that is, with $r=1$) have been catalogued in the {\it OEIS\/}, as seen in Table~\ref{1m}.

The closed form expression $s_{1,1}(n)=2^{n-1}$, as seen, is easily deduced. The closed form expression
\begin{equation}\label{faux}
s_{1,2}(n)=\frac{\alpha^n-\beta^n}{\alpha-\beta},\qquad\alpha=\frac{1+\sqrt5}2,\qquad\beta=\frac{1-\sqrt5}2,
\end{equation}
is known as Binet's formula, i.e.,
\begin{equation}\label{Binet}
F_n=\frac{\alpha^n-\beta^n}{\alpha-\beta}.
\end{equation}
Then~\eqref{faux} may be deduced by means of a generating series, which we
discuss briefly in \S\ref{gen}. 

When $m>2$, it becomes difficult to obtain a closed form for the sequence $s_{r,m}(n)$, as the order of recurrence exceeds~2. This, too, is discussed briefly in \S\ref{gen}. 
\section{Generating functions}\label{gen}
Given a sequence $a(n)$, $n\ge1$, we define its {\bf generating function} by
$$g(x)=\sum_{n=1}^{\infty}a(n)x^n.$$
This series is treated formally. Thus, we may obtain a closed form for the generating function as a means of cataloguing the sequence. The generating function for $s_{r,m}(n)$ is easily obtained. 

\begin{theorem}\label{genrm}
The generating function for $s_{r,m}(n)$ is given by
$$g(x)=\sum_{n=1}^{\infty}s_{r,m}(n)x^n=\frac{x^r}{1-x^r-x^m}.$$
\end{theorem}
\begin{proof}
For ease of notation, let
$$X=x^r+x^{2r}+\cdots+x^{\left[\frac{m}{r}\right]r}.$$
Then, applying Theorem~\ref{rm}, 
\begin{align*}
g(x)&=\sum_{n=1}^{\infty}s_{r,m}(n)x^n\\
&=X+\sum_{n=m+1}^{\infty}s_{r,m}(n-r)x^n+\sum_{n=m+1}^{\infty}s_{r,m}(n-m)x^n\\
&=X+x^r\sum_{n=m+1-r}^{\infty}s_{r,m}(n)x^n+x^m\sum_{n=1}^{\infty}s_{r,m}(n)x^n\\
&=X+x^rg(x)-(X-x^r)+x^mg(x),\end{align*}
where the fourth line follows because
\begin{align*}
x^r\sum_{n=m+1-r}^{\infty}s_{r,m}(n)x^n&=x^r\left(g(x)-\left(x^r+x^{2r}+\cdots+x^{\left[\frac{m-r}{r}\right]r}\right)\right)\\
&=x^rg(x)-\left(x^{2r}+x^{3r}+\cdots+x^{\left[\frac{m}{r}\right]r}\right).\end{align*}
Hence 
$$g(x)=(x^r+x^m)g(x)+x^r.$$
Solving for $g(x)$ yields the result.
\end{proof}

Theorem~\ref{genrm} gives $s_{1,1}(n)=2^{n-1}$ immediately by dint of the geometric sum formula, and then simply by matching coefficients:
$$\sum_{n=1}^{\infty}s_{1,1}(n)x^n=\frac{x}{1-2x}=\sum_{n=1}^{\infty}2^{n-1}x^{n}.$$
Likewise, Theorem~\ref{genrm} gives~\eqref{faux} by using Binet's method. For, partial fraction decomposition yields
$$\frac{x}{1-x-x^2}=\frac1{\sqrt5}\left(\frac1{1-\alpha x}-\frac1{1-\beta x}\right).$$
Thus, by the geometric sum formula,
$$\frac{x}{1-x-x^2}=\frac1{\sqrt5}\sum_{n=0}^{\infty}\alpha^nx^n-\frac1{\sqrt5}\sum_{n=0}^{\infty}\beta^nx^n.$$
We note that $\alpha-\beta=\sqrt5$. Thus by Theorem~\ref{rm}
\begin{equation}\label{Fibx}
\sum_{n=1}^{\infty}s_{1,2}(n)x^n=\sum_{n=1}^{\infty}\frac{\alpha^n-\beta^n}{\alpha-\beta}x^n,
\end{equation}
and we are left to match the coefficients, hence the result.

What happens when we look at $s_{1,3}(n)$? We may factor
$$1-x-x^3=\left((s+t)-x\right)\left((1+(s+t)^2)+(s+t)x+x^2\right),$$
where 
$$s=\sqrt[3]{\frac12+\sqrt{\frac{31}{108}}}\;,\qquad t=\sqrt[3]{\frac12-\sqrt{\frac{31}{108}}}\;.$$
This factorization becomes evident when considering that $3st=-1$ and $s^3+t^3=1$. Letting $u=s+t$, we now write
$$1-x-x^3=\left(u-x\right)\left((1+u^2)+ux+x^2\right).$$

The linear factor $u-x$ introduces a real root of the cubic equation $1-x-x^3=0$, while the quadratic factor $(1+u^2)+ux+x^2$ introduces two complex conjugate roots. Ordinary partial fraction decomposition takes the form
$$\frac{x}{1-x-x^3}=\frac1{1+3u^2}\left(\frac{u}{u-x}-\frac{(1+u^2)-ux}{(1+u^2)+ux+x^2}\right).$$
The geometric sum formula yields
$$\frac{u}{u-x}=\sum_{n=0}^{\infty}\frac{x^n}{u^n}.$$
Then, further partial fraction decomposition using complex numbers yields
\begin{equation}\label{111}
\frac{(1+u^2)-ux}{(1+u^2)+ux+x^2}=\frac{Au\gamma}{1-u\gamma x}+\frac{Bu\beta}{1-u\beta x},
\end{equation}
where, writing $v=s-t$,
$$A=\frac{u}2-\frac{\sqrt3(u^2+v^2)}{4v}i,\quad B=\frac{u}2+\frac{\sqrt3(u^2+v^2)}{4v}i,$$
and
$$\beta=-\frac{u}2-\frac{\sqrt3v}2i,\qquad \gamma=-\frac{u}2+\frac{\sqrt3v}2i.$$
Applying the geometric sum formula to~\eqref{111},
$$\frac{(1+u^2)-ux}{(1+u^2)+ux+x^2}=Au\gamma\sum_{n=0}^{\infty}u^n\gamma^nx^n+Bu\beta\sum_{n=0}^{\infty}u^n\beta^nx^n.$$
Thus, as $1-Au\gamma-Bu\beta=0$, we have
$$\frac{x}{1-x-x^3}=\frac1{1+3u^2}\sum_{n=1}^{\infty}\left(\frac1{u^n}-(A\gamma^{n+1}+B\beta^{n+1})u^{n+1}\right)x^n,$$
hence by Theorem~\ref{rm}, and by matching coefficients,
$$s_{1,3}(n)=\frac{1-\left(A\gamma^{n+1}+B\beta^{n+1}\right)u^{2n+1}}{u^n(1+3u^2)}.$$
This result is rather esoteric in nature, and of not much practical use. Furthermore, for larger cases of~$m$, it is likely impossible to extract a closed formula for $s_{r,m}(n)$, especially in light of Abel's theorem regarding the insolubility by radicals of the general polynomial equation of degree~5 or larger.

However, all is not lost. In \S~\ref{rmksec}, we generalize $s_{r,m}(n)$ a bit more. In this case, we see that the generating function works perfectly well in obtaining a closed form.

\section{The sequence $\boldsymbol{s_{r,m,k}(n)}$}\label{rmksec}
We now fix $k\ge1$ and let $s_{r,m,k}(n)$ denote the number of compositions of~$n$ into exactly~$k$ parts, such that all the parts are congruent to~$r$ modulo~$m$. As before we assume $\gcd(r,m)=1$ and $0<r\le m$. It follows that $s_{r,m,k}(n)=0$ whenever $k>n$. Trivially,
\begin{equation}\label{k1}
s_{r,m,1}(n)=\begin{cases}
1,&\text{if $n\equiv r\pmod{m}$,}\\
0,&\text{otherwise.}\end{cases}
\end{equation}

A nontrivial example is $s_{3,17,4}(63)=20$:
\begin{align*}
63 &= 54 + 3 + 3 + 3 & 63 &= 3 + 54 + 3 + 3 \\
63 &= 37 + 20 + 3 + 3 & 63 &= 3 + 37 + 20 + 3 \\
63 &= 37 + 3 + 20 + 3 & 63 &= 3 + 37 + 3 + 20 \\
63 &= 37 + 3 + 3 + 20 & 63 &= 3 + 20 + 37 + 3 \\
63 &= 20 + 37 + 3 + 3 & 63 &= 3 + 20 + 20 + 20 \\
63 &= 20 + 20 + 20 + 3 & 63 &= 3 + 20 + 3 + 37 \\
63 &= 20 + 20 + 3 + 20 & 63 &= 3 + 3 + 54 + 3 \\
63 &= 20 + 3 + 37 + 3 & 63 &= 3 + 3 + 37 + 20 \\
63 &= 20 + 3 + 20 + 20 & 63 &= 3 + 3 + 20 + 37 \\
63 &= 20 + 3 + 3 + 37 & 63 &= 3 + 3 + 3 + 54 \end{align*}
The lexicographic order in which these compositions are listed suggests a recurrence in the sequence $s_{r,m,k}(n)$. For, considering~\eqref{acform}, it follows that either $a_1=r$ or $a_1>r$. By dint of~\eqref{k1}, we may assume $k>1$. If $n\le m$, then $a_j\le m$ for all~$j$; hence $a_j=r$ for all~$j$. Thus
\begin{equation}\label{rmkinit}
s_{r,m,k}(n)=\begin{cases}
1,&\text{if $n=rk$,}\\
0,&\text{otherwise,}
\end{cases}
\end{equation}
whenever $n\le m$. Similarly,
\begin{equation}\label{rmkinit1}
s_{r,m,k-1}(n)=\begin{cases}
1,&\text{if $n=r(k-1)$,}\\
0,&\text{otherwise,}
\end{cases}
\end{equation}
whenever $n\le m$. 

Otherwise $n>m$. Letting~$A$ denote the set of such compositions where $a_1=r$, and~$B$ the set where $a_1>r$, we see that $s_{r,m,k}=\#A+\#B$. Each member of~$A$ has the form
$$n-r=a_2+a_3+\cdots+a_k,$$
hence $\#A=s_{r,m,k-1}(n-r)$. Likewise, each member of~$B$ has the form
$$n-m=(a_1-m)+a_2+a_3+\cdots+a_k,$$
hence $\#B=s_{r,m,k}(n-m)$. Therefore
\begin{equation}\label{rmkrecur}
s_{r,m,k}(n)=s_{r,m,k-1}(n-r)+s_{r,m,k}(n-m)
\end{equation}
whenever $n>m$. 

Let $g_k(x)$ denote the generating function of the sequence $s_{r,m,k}(n)$. Thus,
\begin{equation}\label{rmkgen1}
g_k(x)=\sum_{n=1}^{m}s_{r,m,k}(n)x^n+\sum_{n=m+1}^{\infty}s_{r,m,k}(n)x^n.
\end{equation}
We remark here, by~\eqref{rmkinit} and~\eqref{rmkinit1} we have
$$\sum_{n=1}^ms_{r,m,k}(n)x^n=\begin{cases}
x^{rk},&\text{if $m\ge rk$,}\\
0,&\text{if $m<rk$,}
\end{cases}$$
and,
$$x^r\sum_{n=1}^{m-r}s_{r,m,k-1}(n)x^n=\begin{cases}
x^{rk},&\text{if $m\ge rk$,}\\
0,&\text{if $m<rk$.}
\end{cases}$$
Therefore
\begin{equation}\label{sumxout}
\sum_{n=1}^{m}s_{r,m,k}(n)x^n-x^r\sum_{n=1}^{m-r}s_{r,m,k-1}(n)x^n=0.
\end{equation}

We may now obtain here the generating function $g_k(x)$. 

\begin{theorem}\label{genrmk}
For all~$k\ge1$ we have
$$g_k(x)=\sum_{n=1}^{\infty}s_{r,m,k}(n)x^n=\frac{x^{rk}}{\left(1-x^m\right)^k}.$$
\end{theorem}
\begin{proof}
Suppose $k>1$. By~\eqref{rmkrecur} we have
\begin{align*}
\sum_{n=m+1}^{\infty}s_{r,m,k}(n)x^n&=\sum_{n=m+1}^{\infty}s_{r,m,k-1}(n-m)x^n+\sum_{n=m+1}^{\infty}s_{r,m,k}(n)x^n\\
&=x^r\sum_{n=m+1-r}^{\infty}s_{r,m,k-1}x^n+x^m\sum_{n=1}^{\infty}s_{r,m,k}(n)x^n,\end{align*}
which, along with~\eqref{rmkgen1} and~\eqref{sumxout}, yields
\begin{align*}
g_k(x)&=\sum_{n=1}^{m}s_{r,m,k}(n)x^n+x^rg_{k-1}(x)-x^r\sum_{n=1}^{m-r}s_{r,m,k-1}(n)x^n+x^mg_k(x)\\
&=x^rg_{k-1}(x)+x^mg_k(x).\end{align*}
Solving for $g_k(x)$, we have
\begin{equation}\label{gkaux}
g_k(x)=\frac{x^{r}}{1-x^m}\,g_{k-1}(x).
\end{equation}
By~\eqref{k1} we have
$$g_1(x)=\sum_{j=0}^{\infty}x^{mj+r}=\frac{x^r}{1-x^m},$$
hence by induction we achieve the desired result.
\end{proof}

We are now able to apply the generating function $g_k(x)$ to obtain a closed form expression for $s_{r,m,k}(n)$. 

\begin{theorem}\label{exact}
For all $k\ge1$, $m\ge1$, $0<r\le m$, $\gcd(r,m)=1$, we have $s_{r,m,k}(n)=0$ if $n\not\equiv rk\pmod{m}$. Otherwise $n\equiv rk\pmod{m}$ and
$$s_{r,m,k}(n)=\binom{\frac{n-rk}{m}+k-1}{k-1}.$$
\end{theorem}
\begin{proof}
We apply Theorem~\ref{genrmk}, and the combinatorial identity
\begin{equation}\label{zcombi}
\frac1{(1-z)^k}=\sum_{j=0}^{\infty}\binom{j+k-1}{k-1}z^j,
\end{equation}
to obtain
$$\sum_{n=1}^{\infty}s_{r,m,k}(n)x^n=g_k(x)=\sum_{j=0}^{\infty}\binom{j+k-1}{k-1}x^{mj+rk}.$$
Therefore by comparing coefficients we have
$$s_{r,m,k}(n)=\binom{j+k-1}{k-1},$$
provided that $n=mj+rk$ for some integer $j\ge 0$; otherwise, $s_{r,m,k}(n)=0$. The statement of the theorem follows immediately.
\end{proof}

Therefore, as mentioned in \S~\ref{intro}, $s_{1,2,3}(73)=666$. Further examples include $s_{17,40,9}(1753)=377348994$, and $s_{5,12,8}(537)=0$, as $537\not\equiv40\pmod{12}$. 

Also, note that~\eqref{zcombi} is proved easily by induction on~$k$.

\section{Pascal's triangle and $\boldsymbol{s_{r,m}(n)}$}\label{srmpascal}
As $s_{r,m,k}(n)$ are binomial coefficients, they may be used to obtain a formula for $s_{r,m}(n)$ from \S~\ref{recerrm}, however artificial it may be. For, by definition it follows that
\begin{equation}\label{putz}
s_{r,m}(n)=\sum_{k=1}^{[n/r]}s_{r,m,k}(n).
\end{equation}
It is clear that we cannot have more than $\left[n/r\right]$ parts, as~$r$ is the smallest possible part. The addends $s_{r,m,k}(n)$ are nonzero if and only if $n\equiv rk\pmod{m}$, that is, $k\equiv r^{-1}n\pmod{m}$, where~$r^{-1}$ denotes the multiplicative inverse of~$r$ modulo~$m$ (which exists as $\gcd(r,m)=1$). Letting~$\xi$ denote the least positive residue of~$r^{-1}n$ modulo~$m$, we have
\begin{equation}\label{k}
k=m\lambda+\xi
\end{equation}
for some nonnegative integer~$\lambda$ in every nonzero addend $s_{r,m,k}(n)$ Thus, as $k\le n/r$, 
\begin{equation}\label{lambda}
\lambda\le\frac{\frac{n}{r}-\xi}{m}=\frac{n-r\xi}{rm}.
\end{equation}
Also from~\eqref{k} we have
\begin{equation}\label{up}
\frac{n-rk}{m}+k-1=(m-r)\lambda+\frac{n-r\xi}{m}+\xi-1.
\end{equation}
Hence by Theorem~\ref{exact}, and by~\eqref{putz}, \eqref{k}, \eqref{lambda}, and~\eqref{up}, 
\begin{equation}\label{tryout}
s_{r,m}(n)=\sum_{\lambda=0}^{\left[\frac{n-r\xi}{rm}\right]}
\binom{(m-r)\lambda+\frac{n-r\xi}{m}+\xi-1}{m\lambda+\xi-1}.
\end{equation}
We may apply~\eqref{tryout} to show that $s_{1,2}(10)=55$. For, $\xi=2$, hence $0\le\lambda\le4$, hence
$$s_{1,2}(10)=\sum_{\lambda=0}^4\binom{\lambda+5}{2\lambda+1}
=\binom{5}{1}+\binom{6}{3}+\binom{7}{5}+\binom{8}{7}+\binom{9}{9}=55.$$
Similarly, we obtain $s_{3,7}(35)=28$: here $\xi=7$, thus $\lambda=0$ by~\eqref{lambda}. Hence
$$s_{3,7}(35)=\binom{8}{6}=28.$$
Similarly, $s_{2,5}(31)=154$: here $\xi=3$, hence $0\le\lambda\le2$. Thus,
$$s_{2,5}(31)=\sum_{\lambda=0}^{2}\binom{3\lambda+7}{5\lambda+2}=
\binom{7}{2}+\binom{10}{7}+\binom{13}{12}=154.$$

Expressing $s_{r,m}(n)$ in this way produces some well known identities involving sums of binomial coefficients. As $s_{1,1}(n)=2^{n-1}$, we obtain
$$2^{n-1}=\sum_{\lambda=0}^{n-1}\binom{n-1}{\lambda},$$
and as $s_{1,2}=F_n$ we obtain
$$F_n=\sum_{\lambda=0}^{\frac{n-1}2}\binom{\lambda+\frac{n-1}2}{2\lambda},\quad\text{if $n$ is odd,}$$
$$F_n=\sum_{\lambda=0}^{\frac{n-2}2}\binom{\lambda+\frac{n}2}{2\lambda+1},\quad\text{if $n$ is even.}$$
For Narayana's cows sequence we obtain
$$s_{1,3}(n)=\sum_{\lambda=0}^{\frac{n-1}3}\binom{2\lambda+\frac{n-1}3}{3\lambda},\quad\text{if $n\equiv1\pmod3$,}$$
$$s_{1,3}(n)=\sum_{\lambda=0}^{\frac{n-2}3}\binom{2\lambda+\frac{n+1}3}{3\lambda+1},\quad\text{if $n\equiv2\pmod3$,}$$
$$s_{1,3}(n)=\sum_{\lambda=0}^{\frac{n-3}3}\binom{2\lambda+\frac{n+3}3}{3\lambda+2},\quad\text{if $n\equiv0\pmod3$,}$$
and for the Padovan sequence,
$$s_{2,3}(n)=\sum_{\lambda=0}^{\left[\frac{n-4}6\right]}\binom{\lambda+\frac{n-1}3}{3\lambda+1},\quad\text{if $n\equiv1\pmod3$,}$$
$$s_{2,3}(n)=\sum_{\lambda=0}^{\left[\frac{n-2}6\right]}\binom{\lambda+\frac{n-2}3}{3\lambda},\quad\text{if $n\equiv2\pmod3$,}$$
$$s_{2,3}(n)=\sum_{\lambda=0}^{\left[\frac{n-6}6\right]}\binom{\lambda+\frac{n}3}{3\lambda+2},\quad\text{if $n\equiv0\pmod3$.}$$
We observe that $s_{r,m}(n)$ is thus a sum of binomial coefficients, taken along a falling diagonal in Pascal's triangle, except that $s_{1,1}(n)$ is a sum taken along a row. This is evident from~\eqref{tryout}. The initial term (when $\lambda=0$), at the top of the diagonal, occurs in column~$\xi-1$ of Pascal's triangle. By definition, we see that $0\le\xi-1\le m-1$, and that~$\xi-1$ is determined completely by the residue of~$n$ modulo~$m$. Thus, the diagonal begins in any one of the leftmost~$m$ columns of Pascal's triangle, determined by~$n$ modulo~$m$. From~\eqref{tryout}, it is also evident that the diagonal will slope downward by jumps of $m-r$ rows down, and~$m$ columns to the right. 
\begin{figure}
\centering
\begin{picture}(360,240)
\put(0,0){\line(1,0){360}}\put(0,0){\line(0,1){240}}
\put(0,240){\line(1,0){360}}\put(360,0){\line(0,1){240}}
\put(30,220){1}
\put(30,200){1}\put(60,200){1}
\put(30,180){1}\put(60,180){2}\put(90,180){1}
\put(30,160){1}\put(60,160){3}\put(90,160){3}\put(120,160){1}
\put(30,140){1}\put(60,140){4}\put(90,140){6}\put(120,140){4}\put(150,140){1}
\put(30,120){1}\put(60,120){5}\put(90,120){10}\put(120,120){10}\put(150,120){5}\put(180,120){1}
\put(30,100){1}\put(60,100){6}\put(90,100){15}\put(120,100){20}\put(150,100){15}\put(180,100){6}\put(210,100){1}
\put(30,80){1}\put(60,80){7}\put(90,80){21}\put(120,80){35}\put(150,80){35}\put(180,80){21}\put(210,80){7}
\put(240,80){1}
\put(30,60){1}\put(60,60){8}\put(90,60){28}\put(120,60){56}\put(150,60){70}\put(180,60){56}\put(210,60){28}
\put(240,60){8}\put(270,60){1}
\put(30,40){1}\put(60,40){9}\put(90,40){36}\put(120,40){84}\put(150,40){126}\put(180,40){126}\put(210,40){84}
\put(240,40){36}\put(270,40){9}\put(300,40){1}
\put(30,20){1}\put(60,20){10}\put(90,20){45}\put(120,20){120}\put(150,20){210}\put(180,20){252}\put(210,20){210}
\put(240,20){120}\put(270,20){45}\put(300,20){10}\put(330,20){1}
\put(33,163){\circle{20}}\put(125,123){\circle{20}}\put(213,83){\circle{20}}\put(303,43){\circle{20}}
\put(180,180){$1+10+7+1=19$}
\end{picture}
 \caption{Illustration for $s_{1,3}(10)=19$.}
 \label{P131}
\end{figure}

\begin{figure}
\centering
\begin{picture}(360,240)
\put(0,0){\line(1,0){360}}\put(0,0){\line(0,1){240}}
\put(0,240){\line(1,0){360}}\put(360,0){\line(0,1){240}}
\put(30,220){1}
\put(30,200){1}\put(60,200){1}
\put(30,180){1}\put(60,180){2}\put(90,180){1}
\put(30,160){1}\put(60,160){3}\put(90,160){3}\put(120,160){1}
\put(30,140){1}\put(60,140){4}\put(90,140){6}\put(120,140){4}\put(150,140){1}
\put(30,120){1}\put(60,120){5}\put(90,120){10}\put(120,120){10}\put(150,120){5}\put(180,120){1}
\put(30,100){1}\put(60,100){6}\put(90,100){15}\put(120,100){20}\put(150,100){15}\put(180,100){6}\put(210,100){1}
\put(30,80){1}\put(60,80){7}\put(90,80){21}\put(120,80){35}\put(150,80){35}\put(180,80){21}\put(210,80){7}
\put(240,80){1}
\put(30,60){1}\put(60,60){8}\put(90,60){28}\put(120,60){56}\put(150,60){70}\put(180,60){56}\put(210,60){28}
\put(240,60){8}\put(270,60){1}
\put(30,40){1}\put(60,40){9}\put(90,40){36}\put(120,40){84}\put(150,40){126}\put(180,40){126}\put(210,40){84}
\put(240,40){36}\put(270,40){9}\put(300,40){1}
\put(30,20){1}\put(60,20){10}\put(90,20){45}\put(120,20){120}\put(150,20){210}\put(180,20){252}\put(210,20){210}
\put(240,20){120}\put(270,20){45}\put(300,20){10}\put(330,20){1}
\put(63,143){\circle{20}}\put(155,103){\circle{20}}\put(243,63){\circle{20}}\put(333,23){\circle{20}}
\put(180,180){$4+15+8+1=28$}
\end{picture}
 \caption{Illustration for $s_{1,3}(11)=28$.}
 \label{P132}
\end{figure}

\begin{figure}
\centering
\begin{picture}(360,240)
\put(0,0){\line(1,0){360}}\put(0,0){\line(0,1){240}}
\put(0,240){\line(1,0){360}}\put(360,0){\line(0,1){240}}
\put(30,220){1}
\put(30,200){1}\put(60,200){1}
\put(30,180){1}\put(60,180){2}\put(90,180){1}
\put(30,160){1}\put(60,160){3}\put(90,160){3}\put(120,160){1}
\put(30,140){1}\put(60,140){4}\put(90,140){6}\put(120,140){4}\put(150,140){1}
\put(30,120){1}\put(60,120){5}\put(90,120){10}\put(120,120){10}\put(150,120){5}\put(180,120){1}
\put(30,100){1}\put(60,100){6}\put(90,100){15}\put(120,100){20}\put(150,100){15}\put(180,100){6}\put(210,100){1}
\put(30,80){1}\put(60,80){7}\put(90,80){21}\put(120,80){35}\put(150,80){35}\put(180,80){21}\put(210,80){7}
\put(240,80){1}
\put(30,60){1}\put(60,60){8}\put(90,60){28}\put(120,60){56}\put(150,60){70}\put(180,60){56}\put(210,60){28}
\put(240,60){8}\put(270,60){1}
\put(30,40){1}\put(60,40){9}\put(90,40){36}\put(120,40){84}\put(150,40){126}\put(180,40){126}\put(210,40){84}
\put(240,40){36}\put(270,40){9}\put(300,40){1}
\put(30,20){1}\put(60,20){10}\put(90,20){45}\put(120,20){120}\put(150,20){210}\put(180,20){252}\put(210,20){210}
\put(240,20){120}\put(270,20){45}\put(300,20){10}\put(330,20){1}
\put(95,83){\circle{20}}\put(185,63){\circle{20}}\put(273,43){\circle{20}}
\put(180,180){$21+56+9=86$}
\end{picture}
 \caption{Illustration for $s_{2,3}(21)=86$.}
 \label{P231}
\end{figure}
Figures~\ref{P131}, \ref{P132}, and~\ref{P231} offer quick illustrations of this phenomenon where $m=3$.

\section{Restrictions on the size of parts}\label{restpart}
Let~$q$ be a fixed natural number. For any natural number~$n$, let $t_q(n)$ denote the number of compositions of~$n$ 
such that all its parts are less than or equal to~$q$. For example, $t_3(5)=13$ as
\begin{align*}
5&=3+2    &5&=1+3+1\\
5&=3+1+1  &5&=1+2+2\\
5&=2+3    &5&=1+2+1+1\\
5&=2+2+1  &5&=1+1+3\\
5&=2+1+2  &5&=1+1+2+1\\
5&=2+1+1+1&5&=1+1+1+2\\
 &        &5&=1+1+1+1+1\end{align*}
Let $u_q(n)$ denote the number of compositions of~$n$ 
such that all its parts are greater than or equal to~$q$. For example, $u_3(11)=13$ as
\begin{align*}
11&=11    &11&=4+7\\
11&=8+3  &11&=4+4+3\\
11&=7+4    &11&=4+3+4\\
11&=6+5  &11&=3+8\\
11&=5+6  &11&=3+5+3\\
11&=5+3+3&11&=3+4+4\\
 &        &11&=3+3+5\end{align*}
As before, we may obtain both the recursion and the generating function for the sequences $t_q(n)$ and $u_q(n)$.

\begin{theorem}\label{trec}
For the fixed natural number~$q$, we have
\begin{equation}\label{t7}
t_q(n)=\begin{cases}
2^{n-1},&\text{if $n\le q$,}\\
\sum_{j=1}^{q}t_q(n-j),&\text{if $n>q$.}
\end{cases}
\end{equation}
\end{theorem}
\begin{proof}
It is clear that $t_q(n)=2^{n-1}$ when $n\le q\,$; this follows directly from Theorem~\ref{01}. Assume, then, that $n>q$. For each~$j$, $1\le j\le q$, let $A_j$ denote the set of compositions counted by $t_q(n)$, such that $a_1=j$ in the notation of~\eqref{acform}. The sets~$A_j$ are mutually disjoint, hence
$$t_q(n)=\sum_{j=1}^q \#A_j.$$
Each composition in~$A_j$ may be written as
$$n-j=a_2+a_3+\cdots+a_k,$$
hence $\#A_j=t_q(n-j)$. The result follows immediately.
\end{proof}

\begin{theorem}\label{tgenf}
The generating function for $t_q(n)$ is given by
\begin{equation*}
f(x)=\sum_{n=1}^{\infty}t_q(n)x^n=\frac{x+x^2+\cdots+x^q}{1-x-x^2-\cdots-x^q}.
\end{equation*}
\end{theorem}
\begin{proof}
By~\eqref{t7} we have
\begin{multline*}
f(x)=x+2x^2+4x^3+\cdots+2^{q-1}x^q\\
+\sum_{n=q+1}^{\infty}t_q(n-1)x^n+\sum_{n=q+1}^{\infty}t_q(n-2)x^n+\cdots+\sum_{n=q+1}^{\infty}t_q(n-q)x^n.
\end{multline*}
Hence,
\begin{align*}
f(x)&=\sum_{n=1}^{q}2^{n-1}x^{n} + \sum_{a=1}^{q-1} x^{a}\left(f(x)-\sum_{b=1}^{q-a}2^{b-1}x^{b}\right)+x^{q}f(x)\\
&=\sum_{n=1}^{q}2^{n-1}x^{n} + \sum_{a=1}^{q}x^{a}f(x) -\sum_{a=1}^{q-1} x^{a}\sum_{b=1}^{q-a}2^{b-1}x^{b}.
\end{align*}
Thus,
\begin{equation}\label{auxtqg}
f(x)(1-x-x^2-\cdots-x^q)=\sum_{n=1}^{q}2^{n-1}x^{n} - \sum_{a=1}^{q-1} x^{a}\sum_{b=1}^{q-a}2^{b-1}x^{b}.
\end{equation}
To simplify the iterated sum, let $n=a+b$. Thus $2\le n\le q$ and $b=n-a$. Thus,
\begin{multline*}
\sum_{a=1}^{q-1} x^{a}\sum_{b=1}^{q-a}2^{b-1}x^{b}=\sum_{n=2}^{q}x^{n}\sum_{a=1}^{n-1}2^{n-a-1}\\
=\sum_{n=2}^{q}x^{n}\left(2^{n-1}-1\right)=\sum_{n=2}^{q}2^{n-1}x^{n}-\sum_{n=2}^{q}x^{n}.
\end{multline*}
Substituting into~\eqref{auxtqg}, 
\begin{align*}
f(x)(1-x-x^2-\cdots-x^q)&=\sum_{n=1}^{q}2^{n-1}x^{n}-\sum_{n=2}^{q}2^{n-1}x^{n}+\sum_{n=2}^{q}x^{n}\\
&=x+x^2+\cdots+x^q,\end{align*}
from which the result follows immediately.
\end{proof}

It is clear by definition that $t_1(n)=1$ for all natural numbers~$n$. Theorem~\ref{tgenf} illustrates this via the geometric sum formula and comparing coefficients:
$$\sum_{n=1}^{\infty}t_1(n)x^n=\frac{x}{1-x}=\sum_{n=1}^{\infty}x^n.$$
Likewise, Theorem~\ref{tgenf}, along with Theorem~\ref{genrm}, \eqref{Binet}, and~\eqref{Fibx}, implies $t_2(n)=F_{n+1}$, as
\begin{gather*}
\sum_{n=1}^{\infty}t_2(n)x^n=\frac{x+x^2}{1-x-x^2}=(1+x)\frac{x}{1-x-x^2}\\
=(1+x)\sum_{n=1}^{\infty}s_{1,2}(n)x^n=(1+x)\sum_{n=1}^{\infty}F_nx^n=\sum_{n=1}^{\infty}F_nx^n+\sum_{n=1}^{\infty}F_nx^{n+1}\\
=x+\sum_{n=1}^{\infty}\left(F_n+F_{n-1}\right)x^n=\sum_{n=1}^{\infty}F_{n+1}x^n.
\end{gather*}
\begin{table}
\centering
 \begin{tabular}{|lll|lll|} 
 \hline
 $q$ & $t_q(n)$ & in {\it OEIS} seq. &  $q$ & $t_q(n)$ & in {\it OEIS} seq.\\
 \hline
 2 & $a(n+1)$ &{\tt A000045} & 7 & $a(n-1)$ & {\tt A172316} \\ 
 3 & $a(n+2)$ &{\tt A000073} & 8 & $a(n-1)$ & {\tt A172317} \\
 4 & $a(n+3)$ &{\tt A000078} & 9 & $a(n-1)$ & {\tt A172318} \\
 5 & $a(n+4)$ &{\tt A001591} & 10 & $a(n-1)$ & {\tt A172319} \\
 6 & $a(n+5)$ &{\tt A001592} & 11 & $a(n-1)$ & {\tt A172320} \\
 \hline
 \end{tabular}
 \caption{Some sequences of form $t_q(n)$.}
 \label{tqno}
\end{table}
Several sequences $t_q(n)$, as defined recursively by~\eqref{t7}, appear in the {\it OEIS\/}, as shown in Table~\ref{tqno}. Obtaining closed form expressions for $t_q(n)$ directly from Theorem~\ref{tgenf} when $q>2$ entails the same difficulties as discussed in \S~\ref{gen}.

\begin{theorem}\label{urec}
For the fixed natural number~$q$, we have
\begin{equation}\label{u7}
u_q(n)=\begin{cases}
0,&\text{if $n< q$,}\\
1,&\text{if $n= q$,}\\
u_q(n-1)+u_q(n-q),&\text{if $n>q$.}
\end{cases}
\end{equation}
\end{theorem}
\begin{proof}
It is clear that $u_q(n)=0$ if $n<q$ by definition, as it is that $u_q(q)=1$. Assume, then, that $n>q$. Let~$A$ denote the set of all compositions counted by~$u_q(n)$ such that $a_1>q$, and let~$B$ denote those in which $a_1=q$, in the notation of~\eqref{acform}. 

As~$A$ and~$B$ are disjoint, we have $u_q(n)\#A+\#B$. Every composition in~$A$ has the form
$$n-1=(a_1-1)+a_2+\cdots+a_k,$$
hence $\#A=u_q(n-1)$. Likewise, every composition in~$B$ has the form
$$n-q=a_2+a_3+\cdots+a_k,$$
hence $\#B=u_q(n-q)$. The result follows immediately.
\end{proof}

\begin{theorem}\label{ugenf}
The generating function for $u_q(n)$ is given by
\begin{equation*}
f(x)=\sum_{n=1}^{\infty}u_q(n)x^n=\frac{x^q}{1-x-x^q}.
\end{equation*}
\end{theorem}
\begin{proof}
By~\eqref{u7} we have
\begin{align*}
f(x)&=x^q+\sum_{n=q+1}^{\infty}u_q(n-1)x^n\sum_{n=q+1}^{\infty}u_q(n-q)x^n\\
&=x^q+xf(x)+x^qf(x),\end{align*}
from which the result follows.
\end{proof}

\begin{table}
\centering
 \begin{tabular}{|lll|lll|} 
 \hline
 $q$ & $u_q(n)$ & in {\it OEIS} seq. &  $q$ & $u_q(n)$ & in {\it OEIS} seq.\\
 \hline
 2 & $a(n-1)$ &{\tt A000045} & 9 & $a(n)\quad$ & {\tt A017903} \\ 
 3 & $a(n)$ &{\tt A078012} & 10 & $a(n)$ & {\tt A017904} \\
 4 & $a(n)$ &{\tt A017898} & 11 & $a(n)$ & {\tt A017905} \\
 5 & $a(n)$ &{\tt A017899} & 12 & $a(n)$ & {\tt A017906} \\
 6 & $a(n)$ &{\tt A017900} & 13 & $a(n)$ & {\tt A017907} \\
 7 & $a(n)$ &{\tt A017901} & 14 & $a(n)$ & {\tt A017908} \\
 8 & $a(n)$ &{\tt A017902} & 15 & $a(n)$ & {\tt A017909} \\
\hline
 \end{tabular}
 \caption{Some sequences of form $u_q(n)$.}
 \label{uqno}
\end{table}

Similar as in our discussion of $t_q(n)$, it is easy to see that $u_1(n)=2^{n-1}$ and $u_2(n)=F_{n-1}$. Several of the sequences $u_q(n)$ also appear in the {\it OEIS\/}, as seen in Table~\ref{uqno}. 

It is interesting to compare $u_3(n)$ with $s_{1,3}(n)$ as discussed in \S~\ref{gen}. The respective generating functions are
$$\sum_{n=1}^{\infty}u_3(n)x^n=\frac{x^3}{1-x-x^3},\qquad\sum_{n=1}^{\infty}s_{1,3}(n)x^n=\frac{x}{1-x-x^3},$$
hence,
$$\sum_{n=1}^{\infty}u_3(n)x^n=\sum_{n=1}^{\infty}s_{1,3}(n)x^{n+2},$$
so that by comparing coefficients we have $u_3(n)=s_{1,3}(n-2)$ for all $n\ge3$ (recall $u_3(1)=u_3(2)=0$).

Now we consider what happens when we restrict the size of the parts from below and from above. Let~$p$ and~$q$ be fixed natural numbers such that $p\le q$. For all natural numbers~$n$, let $v_{p,q}(n)$ denote the number of additive compositions of~$n$ such that all the parts are greater than or equal to~$p$, and less than or equal to~$q$. As an example, we see that $v_{5,9}(16)=6\,$:
\begin{align*}
16&=9+7   &16&=6+5+5\\
16&=8+8   &16&=5+6+5\\
16&=7+9   &16&=5+5+6\end{align*}
The recursion for $v_{p,q}(n)$ is expressed with initial conditions in terms of $u_p(n)$.

\begin{theorem}\label{recurvpq}
For fixed natural numbers $p\le q$, we have $v_{p,q}(n)=u_p(n)$ if $n\le q$. Otherwise $n>q$ whence
$$v_{p,q}(n)=\sum_{k=p}^{q}v_{p,q}(n-k).$$
\end{theorem}
\begin{proof}
If $n\le q$, then no part of any composition of~$n$ exceeds~$q$, therefore $v_{p,q}=u_p(n)$ in this case by definition of $u_p(n)$.

Otherwise $n>q$. In the notation of~\eqref{acform}, each composition counted by $v_{p,q}(n)$ must satisfy $a_1=k$, for some~$k$ such that $p\le k\le q$. Let~$A_k$ denote the set of such compositions such that $a_1=k$. Then the sets~$A_k$ are mutually disjoint, hence 
$$v_{p,q}(n)=\sum_{k=p}^{q}\#A_k.$$
For each~$k$, every composition in~$A_k$ has the form
$$n-k=a_2+a_3+\cdots+a_k,$$
hence $\#A_k=v_{p,q}(n-k)$. The result follows.
\end{proof}                                          

\begin{theorem}\label{toughy}
Let~$p$ and~$q$ be fixed natural numbers such that $p\le q$. The generating function for $v_{p,q}(n)$ is given by
$$f(x)=\sum_{n=1}^{\infty}v_{p,q}(n)=\frac{x^{p}+x^{p+1}+\cdots+x^q}{1-x^{p}-x^{p+1}-\cdots-x^q}.$$
\end{theorem}
\begin{proof}
By Theorem~\ref{recurvpq},
\begin{align*}
f(x)&=\sum_{n=1}^{\infty}v_{p,q}(x)x^n\\
&=\sum_{n=1}^{q}u_p(n)x^n+\sum_{n=q+1}^{\infty}v_{p,q}(n)x^n\\
&=\sum_{n=1}^{q}u_p(n)x^n+\sum_{k=p}^{q}\sum_{n=q+1}^{\infty}v_{p,q}(n-k)x^n.
\end{align*}
Thus by Theorem~\ref{urec},
\begin{equation}\label{vsetup}
f(x)=\sum_{n=p}^{q}u_p(n)x^n+\sum_{k=p}^{q}\sum_{n=q+1}^{\infty}v_{p,q}(n-k)x^n.
\end{equation}
Fix~$k$ such that $p\le k\le q$. Then
\begin{align*}
\sum_{n=q+1}^{\infty}v_{p,q}(n-k)x^n&=x^{k}\sum_{n=q-k+1}^{\infty}v_{p,q}(n)x^n\\
&=x^{k}\left(f(x)-\sum_{n=1}^{q-k}v_{p,q}(n)x^n\right)\\
&=x^{k}f(x)-x^{k}\sum_{n=p}^{q-k}u_p(n)x^n,\end{align*}
where the lower limit of summation in the bottom sum is determined by Theorem~\ref{urec}. Furthermore, this sum is empty when $p>q-k$, thus we need only consider $p\le k\le q-p$ when iterating this sum by summing it over the index~$k$.  We apply this to~\eqref{vsetup} to obtain
$$f(x)=\sum_{n=p}^{q}u_p(n)x^n+f(x)\sum_{k=p}^{q}x^k -\sum_{k=p}^{q-p}x^k\sum_{n=p}^{q-k}u_p(n)x^n,$$
hence
\begin{equation}\label{step2}
f(x)\left(1-x^{p}-x^{p+1}-\cdots-x^{q}\right)=\sum_{n=p}^{q}u_p(n)x^n-\sum_{k=p}^{q-p}x^k\sum_{n=p}^{q-k}u_p(n)x^n.
\end{equation}
This gives two cases for~\eqref{step2}. Either $2p>q$, and the iterated sum on the right is empty, or, $2p\le q$ and the iterated sum is not empty.

In the former case, we note that by Theorem~\ref{urec}, $u_p(n)=1$ for all~$n$ such that $p\le n\le q$, whence
$$f(x)\left(1-x^{p}-x^{p+1}-\cdots-x^{q}\right)=x^p+x^{p+1}+\cdots+x^q,$$
and the statement of the theorem follows immediately. Thus we need only to consider the latter case, $2p\le q$. Here, in the iterated sum in the right hand side of~\eqref{step2}, we let $b=k+n$; thus, as $k+n\le k+(q-k)=q$, we have $2p\le b\le q$. Then, as $n=b-k$, we have
$$\sum_{k=p}^{q-p}x^k\sum_{n=p}^{q-k}u_p(n)x^n=\sum_{b=2p}^{q}x^b\sum_{n=p}^{b-p}u_p(n).$$
Substituting into~\eqref{step2}, and renaming the indices on the right hand side immediately above, we have
\begin{multline}\label{step3}
f(x)\left(1-x^{p}-x^{p+1}-\cdots-x^{q}\right)=\\
\sum_{n=p}^{2p-1}x^n+\sum_{n=2p}^{q}
\left(u_p(n)-\sum_{k=p}^{n-p}u_p(k)\right)x^{n},
\end{multline}
where we recall by Theorem~\ref{urec} that $u_p(n)=1$ for all~$n$ such that $1\le n\le2p-1$. 

It remains to show that
\begin{equation}\label{uno}
u_p(n)-\sum_{k=p}^{n-p}u_p(n)=1
\end{equation}
for all~$n$ such that $2p\le n\le q$. We use induction. When $n=2p$, we have by Theorem~\ref{urec},
$$u_p(2p)-u_p(p)=u_p(2p-1)=1.$$
Thus, for $2p<n\le q$, 
\begin{align*}
u_p(n)-\sum_{k=p}^{n-p}u_p(k)&=u_p(n-1)+u_p(n-p)-\sum_{k=p}^{n-p}u_p(k)\\
&=u_p(n-1)-\sum_{k=p}^{n-p-1}u_p(k)\,=\,1\end{align*}
by induction hypothesis. Thus, having proved~\eqref{uno}, we substitute into~\eqref{step3} to obtain
\begin{align*}
f(x)\left(1-x^{p}-x^{p+1}-\cdots-x^{q}\right)&=\sum_{n=p}^{2p-1}x^n+\sum_{n=2p}^{q}x^n\\
&=x^p+x^{p+1}+\cdots+x^{q},\end{align*}
thus proving the theorem.
\end{proof}

It is clear by definition that 
$$v_{p,p}(n)=\begin{cases}
1,&\text{if $p\mid n$,}\\
0,&\text{if $p\nmid n$.}\end{cases}$$
Thus
$$\sum_{n=1}^{\infty}v_{p,p}(n)x^n=\sum_{n=1}^{\infty}x^{pn}=\frac{x^p}{1-x^{p}},$$
which is in accordance with Theorem~\ref{toughy}. It is also clear by definition that $v_{1,q}(n)=t_{q}(n)$, and we see that these functions have the same generating functions by Theorems~\ref{tgenf} and~\ref{toughy}.

\section{Restrictions on both the size and the number of parts}\label{sizenumber}
We consider the counting functions introduced in \S~\ref{restpart}, and classify them further by restricting the number of parts in their respective compositions. Again, recall the notation of~\eqref{acform}.

Thus, for a fixed natural number~$k$, we define $t_{q,k}(n)$ as the number of compositions of~$n$ into exactly~$k$ parts, such that for all parts~$a_j$, we have $a_j\le q$. We define $u_{q,k}(n)$ as the number of compositions of~$n$ into exactly~$k$ parts, such that for all parts~$a_j$, we have $a_j\ge q$. Finally, we define $v_{p,q,k}(n)$ as the number of compositions of~$n$ into exactly~$k$ parts, such that for all parts~$a_j$, we have $p\le a_j\le q$. 

We include here an example of each counting function.
\vskip 12pt
Here, $t_{3,4}(7)=16$:
\begin{align*}
7&=3+2+1+1 & 7&=2+2+2+1 & 7&=2+1+1+3 & 7&=1+2+2+2\\
7&=3+1+2+1 & 7&=2+2+1+2 & 7&=1+2+2+1 & 7&=1+2+1+3\\
7&=3+1+1+2 & 7&=2+1+3+1 & 7&=1+3+1+2 & 7&=1+1+3+2\\
7&=2+3+1+1 & 7&=2+1+2+2 & 7&=1+2+3+1 & 7&=1+1+2+3
\end{align*}

Here, $u_{3,5}(17)=15$:
\begin{align*}
17&=5+3+3+3+3 & 17&=3+5+3+3+3 & 17&=3+3+4+4+3 \\ 
17&=4+4+3+3+3 & 17&=3+4+4+3+3 & 17&=3+3+4+3+4 \\
17&=4+3+4+3+3 & 17&=3+4+3+4+3 & 17&=3+3+3+5+3 \\
17&=4+3+3+4+3 & 17&=3+4+3+3+4 & 17&=3+3+3+4+4 \\
17&=4+3+3+3+4 & 17&=3+3+5+3+3 & 17&=3+3+3+3+5
\end{align*}

Here, $v_{4,7,3}(16)=18$:
\begin{align*}
16&=7+6+3 & 16&=6+5+5 & 16&=5+4+7 \\ 
16&=7+5+4 & 16&=6+4+6 & 16&=4+7+5 \\
16&=7+4+5 & 16&=6+3+7 & 16&=4+6+6 \\
16&=7+3+6 & 16&=5+7+4 & 16&=4+5+7 \\
16&=6+7+3 & 16&=5+6+5 & 16&=3+7+6 \\
16&=6+6+4 & 16&=5+5+6 & 16&=3+6+7
\end{align*}

We may obtain the generating function for $t_{q,k}(n)$ first by introducing an auxiliary notation. We may expand the polynomial $(1+x+\cdots+x^{q-1})^{k}$, and write this as
\begin{equation}\label{expand}
(1+x+\cdots+x^{q-1})^{k}=\sum_{n=0}^{(q-1)k}\left[\begin{matrix}k\\n\end{matrix}\right]_{q}x^n.
\end{equation}
Thus the symbol
$$\left[\begin{matrix}k\\n\end{matrix}\right]_{q}$$
denotes the coefficient of~$x^n$ in the expansion of $(1+x+\cdots+x^{q-1})^{k}$. This, of course, gives a generalization of the binomial theorem, which is the case when $q=2$; i.e.,
$$\left[\begin{matrix}k\\n\end{matrix}\right]_{2}=\binom{k}{n}.$$
The case when $q=1$ is trivial, as both sides in~\eqref{expand} become unity and we obtain
$$\left[\begin{matrix}k\\n\end{matrix}\right]_{1}=\begin{cases}
1,&\text{if $n=0$,}\\
0,&\text{if $n>0$.}\end{cases}$$
The coefficients in~\eqref{expand} may be expressed in terms of binomial coefficients:
\begin{equation}\label{comptet}
\left[\begin{matrix}k\\n\end{matrix}\right]_{q}=\sum_{j=0}^{\left[\frac{n}{q}\right]}(-1)^{j}\binom{k}{j}\binom{k-1+n-qj}{k-1}.
\end{equation}
The formula~\eqref{comptet} may be proved several ways. Perhaps the purest combinatorial proof can be described as follows: the sought coefficient, on the left hand side of~\eqref{comptet}, is the number of ways to distribute~$n$ balls into~$k$ bins, such that no bin may contain more than~$q$ balls. We attain this number as an alternating sum, by applying the inclusion-exclusion principle. We first compute the number of ways to distribute~$n$ balls into~$k$ bins, but, with no restriction on the number of balls placed in any given bin; this number is a binomial coefficient, viz.,
$$\binom{k-1+n}{k-1}.$$
If $n<b$ we are done. If not, then we remove~$b$ balls, so that~$n-b$ balls remain. We now place the~$b$ balls into any one bin, and and distribute the remaining~$n-b$ balls among all the bins. There are 
$$\binom{k}{1}\binom{k-1+n-b}{k-1}$$
to do this, as there are $\binom{k}{1}$ ways of choosing a bin. This quantity is subtracted from our original sum. If $n\ge2q$, we must continue. We now remove~$2b$ balls from the original~$n$ balls, place~$b$ balls each into any given two bins, and distribute the remaining $n-2b$ balls into all the bins. There are $\binom{k}{2}$ ways of choosing two bins, hence we now add back
$$\binom{k}{2}\binom{k-1+n-2b}{k-1}.$$
Then, if $n\ge3b$, we must subtract
$$\binom{k}{3}\binom{k-1+n-3b}{k-1},$$
and so on, through exactly $\left[n/b\right]$ steps. This process yields the alternating sum in the right hand side of~\eqref{comptet}. Another, more analytic, way of proving~\eqref{comptet} is to rewrite~\eqref{expand} as
$$\left(\frac{1-x^{q}}{1-x}\right)^{k}=\sum_{n=0}^{(q-1)k}\left[\begin{matrix}k\\n\end{matrix}\right]_{q}x^n.$$
Applying~\eqref{zcombi},
$$\sum_{n=0}^{(q-1)k}\left[\begin{matrix}k\\n\end{matrix}\right]_{q}x^n=\left(\sum_{n=0}^{\infty}\binom{n+k-1}{k-1}x^{n}\right)
\left(\sum_{n=0}^{k}\binom{k}{n}(-1)^{n}x^{qn}\right).$$
Both polynomials in this equation may be treated formally as infinite sums; i.e., the coefficients $\left[\begin{smallmatrix}k\\n\end{smallmatrix}\right]_{q}$, resp. $\binom{k}{n}$, are zero if $n>(q-1)k$, resp. $k>n$. Then we take the Cauchy product on the right hand side, and then compare coefficients, thus obtaining~\eqref{comptet}. 

Having said all this, it suffices to express~$t_{q,k}(n)$ in terms of the coefficients introduced in~\eqref{expand}. 

\begin{theorem}\label{tqkall}
For all fixed natural numbers~$q$ and~$k$, and for all~$n$ such that $k\le n\le qk$, we have
\begin{equation}\label{thmtqkc}
t_{q,k}(n)=\left[\begin{matrix}k\\n-k\end{matrix}\right]_{q}.
\end{equation}
Otherwise $n<k$ or $n>qk$, whence $t_{q,k}(n)=0$. Furthermore, the generating function for $t_{q,k}(n)$ is
\begin{equation}\label{thmtqksecond}
f(x)=\sum_{n=k}^{qk}t_{q,k}(n)=x^{k}\left(1+x+\cdots+x^{q-1}\right)^{k}.
\end{equation}
\end{theorem}
\begin{proof}
We begin by considering the quantity
\begin{equation}\label{auxtqkc}
\left[\begin{matrix}k\\n-k\end{matrix}\right]_{q},
\end{equation}
that is, the coefficient of $x^{n-k}$ in the expansion of $\left(1+x+\cdots+x^{q-1}\right)^{k}$. To compute this quantity, we must count all possible ways of choosing one term, of the form $x^{b}$, from each of the~$k$ identical factors $1+x+\cdots+x^{q-1}$, and then multiply these terms together, thus obtaining an expression of the form
$$x^{b_1}x^{b_2}\cdots x^{b_k}=x^{b_1+b_2+\cdots+b_k},$$
where $b_1+b_2+\cdots+b_k=n-k$ and $0\le b_j\le q-1$, $1\le j\le k$. Then,
$$n=a_1+a_2+\cdots+a_k, \qquad 1\le a_{j}\le q,\qquad 1\le j\le k.$$
Hence the coefficient~\eqref{auxtqkc} counts the number of compositions of~$n$ into~$k$ parts, each part~$a$ having the property $a\le q$. This proves~\eqref{thmtqkc}. By definition it is clear that $t_{q,k}=0$ if $n<k$ or $n>qk$. 

Finally, by~\eqref{expand} and~\eqref{thmtqkc} we have
$$f(x)=\sum_{n=k}^{qk}\left[\begin{matrix}k\\n-k\end{matrix}\right]_{q}x^{n}=x^{k}\sum_{n=0}^{(q-1)k}\left[\begin{matrix}k\\n
\end{matrix}\right]_{q}x^{n}=x^{k}\left(1+x+\cdots+x^{q-1}\right)^{k},$$
thus proving~\eqref{thmtqksecond}.
\end{proof}

\begin{theorem}\label{uqkbigres}
For all fixed natural numbers~$q$ and~$k$, and for all~$n$ such that $n\ge qk$, we have
\begin{equation}\label{pluginhere}
u_{q,k}(n)=\binom{n-qk+k-1}{k-1},
\end{equation}
otherwise $n<qk$ and $u_{q,k}(n)=0$.

Furthermore, the generating function for $u_{q,k}(n)$ is
\begin{equation}\label{plugin2}
f(x)=\sum_{n=qk}^{\infty}u_{q,k}(n)x^n=\frac{x^{qk}}{\left(1-x\right)^{k}}.
\end{equation}
\end{theorem}
\begin{proof}
In the notation of~\eqref{acform}, let $n=a_1+a_2+\cdots+a_k$ be a composition of~$n$ that is counted by $u_{q,k}(n)$, so that $a_k\ge q$ for all~$j$, $0\le j\le k$. Now let $b_j=a_j-(q-1)$, $1\le j\le k$. Thus $b_j\ge1$, $1\le j\le k$, hence,
$$
n-k(q-1)=b_1+b_2+\cdots+b_k.
$$
Thus 
$$
u_{q,k}(n)=u_{1,k}\left(n-k(q-1)\right).
$$
By definition, recalling the notation of \S~\ref{rmksec},
$$
u_{1,k}(n-k(q-1))=s_{1,1,k}\left(n-k(q-1)\right),
$$
hence \eqref{pluginhere} follows by Theorem~\ref{exact}. By definition, it is clear that $u_{k,q}(n)=0$ if $n<qk$. 

We may now obtain~\eqref{plugin2} by applying~\eqref{pluginhere} and~\eqref{zcombi}:
$$f(x)=\sum_{n=qk}^{\infty}\binom{n-qk+k-1}{k-1}x^n
=x^{qk}\sum_{n=0}^{\infty}\binom{n+k-1}{k-1}x^n
=\frac{x^{qk}}{(1-x)^{k}}.$$
\end{proof}

\begin{theorem}\label{vpqfin}
For all fixed natural numbers~$p$, $q$ and~$k$, such that $p\le q$, and for all~$n$ such that $pk\le n\le qk$, we have
\begin{equation}\label{lastplug}
v_{p,q,k}(n)=\left[\begin{matrix}k\\n-pk\end{matrix}\right]_{q-p+1}.
\end{equation}
Otherwise either $n<pk$ or $n>qk$, whence $v_{p,q,k}(n)=0$. Furthermore, the generating function for $v_{p,q,k}(n)$ is
\begin{equation}\label{lastplug22}
f(x)=\sum_{n=pk}^{qk}v_{p,q,k}(n)x^{n}=x^{pk}\left(1+x+\cdots+x^{q-p}\right)^{k}.
\end{equation}
\end{theorem}
\begin{proof}
This theorem merely generalizes Theorem~\ref{tqkall}, hence the proof is similar. The coefficient of $x^{n-pk}$ in the expansion of $\left(1+x+\cdots+x^{q-p}\right)^{k}$ is as given in the right hand side of~\eqref{lastplug}, 
and hence is computed by counting all possible ways of choosing one term, of the form $x^{b}$, from each of the~$k$ identical factors $1+x+\cdots+x^{q-p}$, and then multiplying these terms together to obtain an expression of the form
$$x^{b_1}x^{b_2}\cdots x^{b_k}=x^{b_1+b_2+\cdots+b_k},$$
where $b_1+b_2+\cdots+b_k=n-kp$ and $0\le b_j\le q-p$, $1\le j\le k$. Again, letting $a_{j}=b_j+p$, we have
$$n=a_1+a+2+\cdots+a_k,\qquad p\le a_j\le q,\qquad1\le j\le k.$$
This proves~\eqref{lastplug}. It is clear by definition that $v_{p,q,k}(n)=0$ if $n<pk$ or $n>qk$.

Finally, by~\eqref{expand} and~\eqref{lastplug} we have
\begin{multline*}
f(x)=\sum_{n=pk}^{qk}\left[\begin{matrix}k\\n-pk\end{matrix}\right]_{q-p+1}x^{n}
=x^{pk}\sum_{n=0}^{(q-p)k}
\left[\begin{matrix}k\\n\end{matrix}\right]_{q-p+1}x^{n}\\
=x^{pk}\left(1+x+\cdots+x^{q-p}\right)^{k},
\end{multline*}
thus proving~\eqref{lastplug22}.
\end{proof}

We leave with three examples from this Section:
\begin{align*}
t_{8,11}(44)&=346718362\\
u_{6,10}(84)&=10015005\\
v_{5,11,13}(86)&=233197198\end{align*}

\section{Concluding remarks}\label{end}
As this note is purely expository, the results presented herein are known within one context or another. Heubach and Mansour~\cite{heuman} have counted the number of compositions of~$n$ such that the parts belong to a given set~$A$ of natural numbers. For example, in the notation used in \S~\ref{restpart}, $v_{4,11}(n)$ counts the number of compositions of~$n$ with parts belonging to the set
$$A=\{4,\,5,\,6,\,7,\,8,\,9,\,10,\,11\}.$$
In the notation used in \S~\ref{recerrm}, $s_{3,7}(n)$ counts the number of compositions of~$n$ with parts belonging to the set
$$A=\{3,\,10,\,17,\,24,\,31,\,38,\,45,\,52,\,\dots\}=\left\{7n+3\right\}_{n\ge0}.$$
This note's purpose is to introduce the reader to compositions, and to some elementary methods used to count them. We suggest the paper of Heubach and Mansour~\cite{heuman} for a deeper perspective into this topic. As a {\tt .pdf} file, it may be downloaded directly from the {\it Wikipedia\/} page for compositions. Be aware, however, that these authors include~0 in the domain of the counting functions of the compositions, whereas in this note, the domain is strictly the set of natural numbers. This produces generating functions that differ slightly from ours. 

Malandro~\cite{malandro} has given an extensive analysis of the compositions counted by~$t_q(n)$ in \S~\ref{restpart}.

In this note, we did not count compositions of~$n$ into distinct parts. Richmond and Knopfmacher~\cite{richmond} have provided detailed analysis on this topic. 

There are many fine textbooks in combinatorics and discrete mathematics that discuss the topic of generating functions. Perhaps the most extensive treatment of this topic is by Wilf~\cite{wilf}. The work cited herein is the third edition, but the author was able, quite easily, to download the second edition as a {\tt .pdf} file via an internet search.

Naturally, we also recommend the {\it OEIS\/}~\cite{oeis} as a valuable resource for integer sequences in general, and as an excellent source for references.

\vskip 24pt\noindent
{\sc University of the Virgin Islands}\\
{\sc 2 John Brewers Bay}\\
{\sc St. Thomas VI 00802}\\
{\sc USA}\\
{\tt diannuc@uvi.edu}
\end{document}